\pdfoutput=1   
    \documentclass[11pt,twoside]{article}

    \usepackage{graphicx}
    \usepackage{styleset}
    \usepackage{macros}
    \title  {Instanton Floer homology and the Alexander polynomial}

    \author {P. B. Kronheimer and T. S. Mrowka%
            \thanks{%
            The work of the first author
            was supported by the National Science Foundation through
            NSF grant number DMS-0405271. The work of
            the second author was supported by NSF grants DMS-0206485,
            DMS-0244663 and DMS-0805841.}} 

    \address {Harvard University, Cambridge MA 02138 \\
              Massachusetts Institute of Technology, Cambridge MA 02139}       

\begin{document}

\maketitle

\begin{abstract}
   The instanton Floer homology of a knot in $S^{3}$ is a vector space
   with a canonical mod $2$ grading. It carries a distinguished endomorphism of
   even degree, arising from the $2$-dimensional homology class
   represented by a Seifert surface. The Floer homology decomposes as
   a direct sum of the generalized eigenspaces of this endomorphism.
   We show that the Euler characteristics of these generalized
   eigenspaces are the coefficients of the Alexander polynomial of the
   knot. Among other applications, we deduce that instanton homology
   detects fibered knots.
\end{abstract}

\section{Introduction}

For a knot $K\subset S^{3}$, the authors defined in
\cite{KM-sutures} a Floer homology group $\KHI(K)$, by a slight
variant of a construction that appeared first in \cite{Floer-Durham-paper}.
In brief, one takes the knot complement $S^{3} \setminus N^{\circ}(K)$
and forms from it a closed $3$-manifold $Z(K)$ by attaching to
$\partial N(K)$ the manifold $F\times S^{1}$, where $F$ is a genus-1
surface with one boundary component. The attaching is done in such a
way that $\{\text{point}\}\times S^{1}$ is glued to the meridian of
$K$ and $\partial F \times \{\text{point}\}$ is glued to the
longitude. The vector space $\KHI(K)$ is then defined by applying
Floer's instanton homology to the closed 3-manifold $Z(K)$. We will
recall the details in section~\ref{sec:background}. If $\Sigma$ is a
Seifert surface for $K$, then there is a corresponding closed surface
$\bar\Sigma$ in $Z(K)$, formed as the union of $\Sigma$ and one copy
of $F$. The homology class  $\bar\sigma=[\bar\Sigma]$ in $H_{2}(Z(K))$ 
determines an
endomorphism $\mu(\bar\sigma)$ on the instanton homology of $Z(K)$,
and hence also an endomorphism of $\KHI(K)$. As was shown in
\cite{KM-sutures}, and as we recall below, the generalized
eigenspaces of $\mu(\bar\sigma)$ give a direct sum decomposition,
\begin{equation}\label{eq:eigenspace-decomposition}
           \KHI(K) = \bigoplus_{j=-g}^{g} \KHI(K,j).
\end{equation}
Here $g$ is the genus of the Seifert surface. In this paper, we will
define a canonical $\Z/2$ grading on $\KHI(K)$, and hence on each 
$\KHI(K,j)$, so that we may write
\[
            \KHI(K,j) = \KHI_{0}(K,j) \oplus \KHI_{1}(K,j).
\]
This allows us to define the Euler characteristic $\chi(\KHI(K,j))$ as
the difference of the ranks of the even and odd parts. The main result
of this paper is the following theorem.

\begin{theorem}\label{thm:main}
    For any knot in $S^{3}$,
    the Euler characteristics $\chi(\KHI(K,j))$ of the summands
    $\KHI(K,j)$ are minus the coefficients of
    the symmetrized Alexander polynomial $\Delta_{K}(t)$, with Conway's
    normalization. That is,
    \[
                    \Delta_{K}(t) = - \sum_{j} \chi(\KHI(K,j)) t^{j}.
    \]
\end{theorem}

The Floer homology group $\KHI(K)$ is supposed to be an  ``instanton''
counterpart to the Heegaard knot homology of Ozsv\'ath-Szab\'o and
Rasmussen \cite{Ozsvath-Szabo-knotfloer,Rasmussen-thesis}. It is known
that the Euler characteristic of Heegaard knot homology gives the
Alexander polynomial; so the above
theorem can be taken as further evidence that the two theories are
indeed closely related.

\begin{figure}
    \begin{center}
        \includegraphics[scale=0.6]{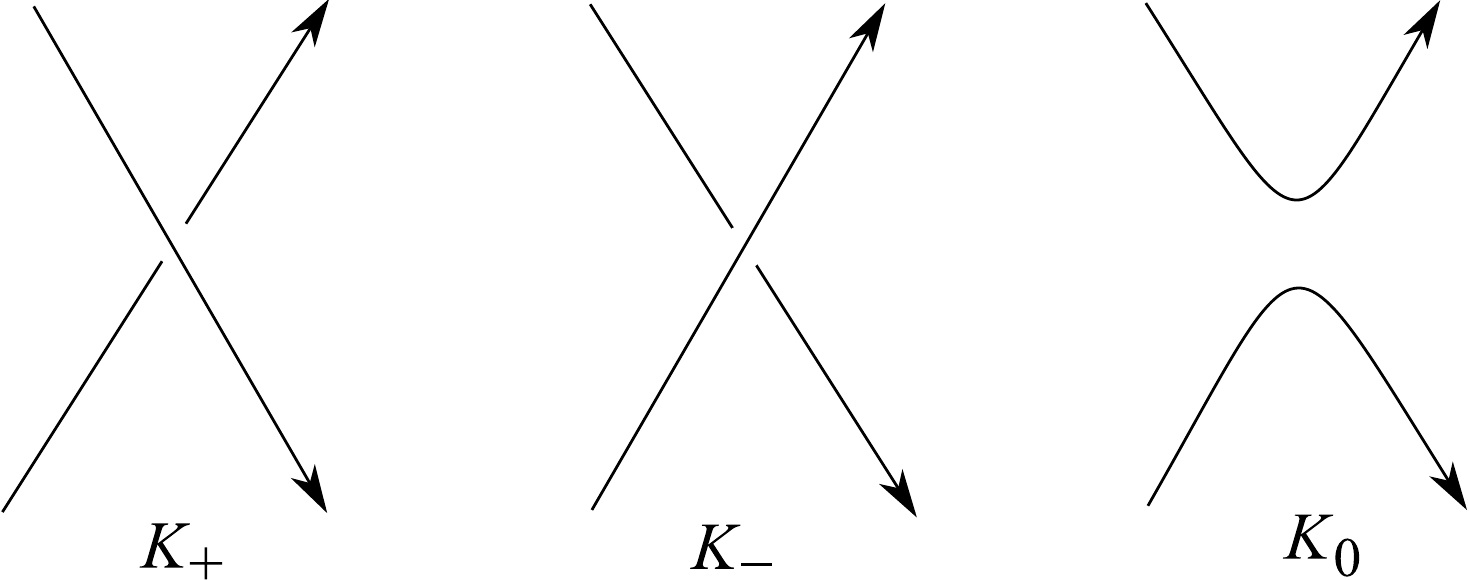}
    \end{center}
    \caption{\label{fig:Oriented-Skein}
    Knots $K_{+}$, $K_{-}$ and $K_{0}$ differing at a single crossing.}
\end{figure}

The proof of the theorem rests on Conway's skein relation for the
Alexander polynomial. To exploit the skein relation in this way, we
first extend the definition of $\KHI(K)$ to links. Then, given three
oriented knots or links $K_{+}$, $K_{-}$ and $K_{0}$ related by the
skein moves (see Figure~\ref{fig:Oriented-Skein}), we establish a long
exact sequence relating the instanton knot (or link) homologies of 
$K_{+}$, $K_{-}$ and $K_{0}$. More precisely,
if for example $K_{+}$ and $K_{-}$ are knots and $K_{0}$ is a
$2$-component link, then we will show that there is along exact
sequence 
\[
           \cdots \to \KHI(K_{+}) \to \KHI(K_{-}) \to \KHI(K_{0}) \to \cdots .
\]
(This situation is a little different when $K_{+}$ and $K_{-}$ are
$2$-component links and $K_{0}$ is a knot: see Theorem~\ref{thm:skein}.)

Skein exact sequences of
this sort for $\KHI(K)$ are not new. The definition of $\KHI(K)$
appears almost verbatim in Floer's  paper
\cite{Floer-Durham-paper}, along with outline
proofs of just such a skein sequence. See in particular part ($2'$) of
Theorem 5 in \cite{Floer-Durham-paper}, which corresponds to
Theorem~\ref{thm:skein} in this paper.
The material of Floer's paper
\cite{Floer-Durham-paper} is also presented in \cite{Braam-Donaldson}. The
proof of the skein exact sequence which we shall describe is essentially
Floer's argument, as amplified in \cite{Braam-Donaldson},
though we shall present it in the context of sutured
manifolds. The new ingredient however is the decomposition
\eqref{eq:eigenspace-decomposition} of the instanton Floer homology,
without which one cannot arrive at the Alexander
polynomial.

\medskip The structure of the remainder of this paper is as
follows. In section~\ref{sec:background}, we recall the construction
of instanton knot homology, as well as instanton homology for sutured
manifolds, following \cite{KM-sutures}. We take the opportunity here
to extend and slightly generalize our earlier results concerning these
constructions. Section~\ref{sec:skein} presents the proof of the main
theorem. Some applications are discussed in section~\ref{sec:applications}. 
The relationship between $\Delta_{K}(t)$ and the instanton homology of
$K$ was conjectured in \cite{KM-sutures}, and the result provides the
missing ingredient to show that the $\KHI$ detects fibered knots. 
Theorem~\ref{thm:main} also provides a lower bound for the rank of the
instanton homology group:

\begin{corollary}\label{cor:alexander-vs-rank}
    If the Alexander polynomial of $K$ is $\sum_{-d}^{d}a_{j} t^{j}$,
    then
    the rank of $\KHI(K)$ is not less than $\sum_{-d}^{d}|a_{j}|$.
    \qed
\end{corollary}

The corollary can be used to draw conclusions about the existence of
certain representations of the knot group in $\SU(2)$. 

\subparagraph{Acknowledgment.}
As this paper was being completed, the authors learned that
essentially the same result has been obtained simultaneously by Yuhan
Lim \cite{Lim}. The authors are grateful to the referee for pointing
out the errors in an earlier version of this paper, particularly
concerning the mod $2$ gradings.

\section{Background}
\label{sec:background}
\subsection{Instanton Floer homology}

Let $Y$ be a closed, connected, oriented $3$-manifold, and let $w\to Y$ be a
hermitian line bundle with the property that the pairing of $c_{1}(w)$
with some class in $H_{2}(Y)$ is odd. If $E\to Y$ is a $U(2)$ bundle
with $\Lambda^{2}E \cong w$, we write $\bonf(Y)_{w}$ for the
space of $\PU(2)$ connections in the adjoint bundle $\ad(E)$, modulo
the action of the gauge group consisting of automorphisms of $E$ with
determinant $1$. The instanton Floer homology group $I_{*}(Y)_{w}$ is
the Floer homology arising from the Chern-Simons functional on
$\bonf(Y)_{w}$. It has a relative grading by $\Z/8$.
Our notation for this Floer group follows
\cite{KM-sutures}; an exposition of its construction is in
\cite{Donaldson-book}. We will always use complex coefficients, so
$I_{*}(Y)_{w}$ is a complex vector space.

If $\sigma$ is a $2$-dimensional integral homology class in $Y$, then
there is a corresponding operator $\mu(\sigma)$ on $I_{*}(Y)_{w}$ of
degree $-2$. If $y\in Y$ is a point representing the generator of
$H_{0}(Y)$, then there is also a degree-$4$ operator $\mu(y)$.
The operators $\mu(\sigma)$, for $\sigma\in H_{2}(Y)$, commute with
each other and with $\mu(y)$.
As shown in \cite{KM-sutures} based on the calculations
of \cite{Munoz}, the simultaneous eigenvalues of the commuting
pair of operators
$(\mu(y),\mu(\sigma))$ all have the form
\begin{equation}\label{eq:eigenvalue-pairs}
         (2, 2k) \qquad\text{or}\qquad  (-2,  2k\sqrt{-1}),
\end{equation}
for even integers $2k$ in the range
\[
                | 2k | \le |\sigma|.
\]
Here $|\sigma|$ denotes the Thurston norm of $\sigma$, the minimum
value of $-\chi(\Sigma)$ over all aspherical embedded surfaces
$\Sigma$ with $[\Sigma]=\sigma$.

\subsection{Instanton homology for sutured manifolds}

We recall the definition of the instanton Floer homology for a
balanced sutured manifold, as introduced in \cite{KM-sutures} with
motivation from the Heegaard counterpart defined in \cite{Juhasz-1}.
The reader is referred to \cite{KM-sutures} and \cite{Juhasz-1} for
background and details.

Let
$(M,\gamma)$ be a balanced sutured manifold. Its oriented boundary is a union,
\[
            \partial M = R_{+}(\gamma) \cup A(\gamma) \cup
           (- R_{-}(\gamma))
\]
where $A(\gamma)$ is a union of annuli, neighborhoods of the sutures
$s(\gamma)$.  To define the instanton homology group $\SHI(M,\gamma)$
we proceed as follows. Let $([-1,1]\times T,\delta)$ be a product sutured
manifold, with $T$ a connected, oriented surface with boundary. The
annuli $A(\delta)$ are the annuli $[-1,1]\times \partial T$, and we
suppose these are in one-to-one correspondence with the annuli
$A(\gamma)$. We attach this product piece to $(M,\gamma)$ along the
annuli to obtain a manifold
\begin{equation}\label{eq:barM}
            \bar{M} = M \cup \bigl( [-1,1]\times T \bigr).
\end{equation}
We write
\begin{equation}\label{eq:boundary-barM}
          \partial  \bar{M} = \bar{R}_{+} \cup (-\bar{R}_{-}).
\end{equation}
We can regard $\bar{M}$ as a sutured manifold (not balanced, because it
has no sutures). The surface $\bar{R}_{+}$ and $\bar{R}_{-}$ are both
connected and are diffeomorphic. We choose an orientation-preserving
diffeomorphism
\[
            h : \bar{R}_{+} \to \bar{R}_{-}
\]
and then define $Z=Z(M,\gamma)$ as the quotient space
\[
           Z = \bar{M}/\sim,
\]
where $\sim$ is the identification defined by $h$. The two surfaces
$\bar{R}_{\pm}$ give a single closed surface
\[
            \bar{R}\subset Z.
\]

We need to impose a side condition on the choice of $T$ and $h$ in
order to proceed. We require that there is a closed curve $c$ in $T$
such that $\{1\}\times c$ and $\{-1\}\times c$ become non-separating
curves in $\bar{R}_{+}$ and $\bar{R}_{-}$ respectively; and we require
further that $h$ is chosen so as to carry $\{1\}\times c$ to
$\{-1\}\times c$ by the identity map on $c$.

\begin{definition}
    We say that $(Z,\bar{R})$ is an admissible closure of $(M,\gamma)$
    if it arises in this way, from some choice of $T$ and $h$,
    satisfying the above conditions. \CloseDef
\end{definition}

\begin{remark}
  In \cite[Definition~4.2]{KM-sutures}, there was an additional
  requirement that $\bar{R}_{\pm}$ should have genus $2$ or more. This
    was needed only in the context there of Seiberg-Witten Floer
    homology, as explained in section~7.6 of
    \cite{KM-sutures}. Furthermore, the notion of closure in 
     \cite{KM-sutures} did not require that $h$ carry $\{1\}\times c$
     to $\{-1\}\times c$, hence the qualification ``admissible'' in the
   present paper. 
\end{remark}

In an admissible closure, the curve $c$ gives rise to a torus $S^{1}\times c$
in $Z$ which meets $\bar{R}$ transversely in a circle. Pick a point
$x$ on $c$. The
closed curve $S^{1}\times \{x\}$ lies on the torus $S^{1}\times c$ and
meets $\bar{R}$ in a
single point. We write
\[
            w \to Z
\]
for a hermitian line bundle on $Z$ whose first Chern class is dual to
$S^{1}\times\{x\}$. Since $c_{1}(w)$ has odd evaluation on the closed
surface $\bar{R}$, the instanton homology group $I_{*}(Z)_{w}$ is
well-defined. As in \cite{KM-sutures}, we write
\[
            I_{*}(Z|\bar{R})_{w} \subset I_{*}(Z)_{w}
\]
for the simultaneous generalized eigenspace of the pair of operators
\[(\mu(y),\mu(\bar{R}))\] belonging to the eigenvalues $(2,2g-2)$, where
$g$ is the genus of $\bar{R}$. (See \eqref{eq:eigenvalue-pairs}.)

\begin{definition}
    For a balanced sutured manifold $(M,\gamma)$, 
   the instanton Floer homology group $\SHI(M,\gamma)$ is defined to
    be $I_{*}(Z|\bar{R})_{w}$, where $(Z,\bar{R})$ is any admissible
    closure of $(M,\gamma)$. \CloseDef.
\end{definition}

It was shown in \cite{KM-sutures} that $\SHI(M,\gamma)$ is
well-defined, in the sense that any two choices of $T$ or $h$ will
lead to isomorphic versions of $\SHI(M,\gamma)$.

\subsection{Relaxing the rules on $T$}
\label{subsec:disconnected-T}

As stated, the definition of $\SHI(M,\gamma)$ requires that we form a
closure $(Z,\bar{R})$ using a \emph{connected} auxiliary surface $T$.
We can relax this condition on $T$, with a little care, and the extra
freedom gained will be convenient in later arguments.

So let $T$ be a possibly disconnected, oriented surface with boundary.
The number of boundary components of $T$ needs to be equal to the
number of sutures in $(M,\gamma)$. We then need to choose an
orientation-reversing
diffeomorphism between $\partial T$ and $\partial R_{+}(\gamma)$, so
as to be able to form a manifold $\bar{M}$ as in \eqref{eq:barM},
gluing $[-1,1]\times \partial T$ to the annuli $A(\gamma)$. We
continue to write $\bar{R}_{+}$, $\bar{R}_{-}$ for the ``top'' and
``bottom'' parts of the boundary of $\partial \bar{M}$, as at
\eqref{eq:boundary-barM}. Neither of these need be connected, although
they have the same Euler number. We shall impose the following
conditions.
\begin{enumerate}
    \item On each connected component $T_{i}$ of $T$, there is an
    oriented
    simple closed curve $c_{i}$ such that the corresponding curves
    $\{1\}\times c_{i}$ and $\{-1\}\times c_{i}$ are both
    non-separating on $\bar{R}_{+}$ and $\bar{R}_{-}$ respectively.

    \item \label{item:T-condition-2} There exists a diffeomorphism $h :
    \bar{R}_{+}\to\bar{R}_{-}$ which carries $\{1\}\times c_{i}$ to
    $\{-1\}\times c_{i}$ for all $i$, as oriented curves.

    \item There is a $1$-cycle $c'$ on $\bar{R}_{+}$ which intersects
    each curve $\{1\}\times c_{i}$ once.
\end{enumerate}
We then choose any $h$ satisfying \ref{item:T-condition-2} and use $h$
to identify the top and bottom, so forming a closed pair $(Z,\bar{R})$ as
before. The surface $\bar{R}$ may have more than one component (but no
more than the number of components of $T$). No component of $\bar{R}$
is a sphere, because each component contains a non-separating curve.
We may regard $T$ as a codimension-zero submanifold of $\bar{R}$ via
the inclusion of $\{1\}\times T$ in $\bar{R}_{+}$.

For each component $\bar{R}_{k}$ of $\bar{R}$, we now choose one
corresponding component $T_{i_{k}}$ of $T\cap\bar{R}_{k}$. We take
$w\to Z$
to be the complex line bundle with $c_{1}(w)$ dual to the sum of the
circles $S^{1}\times \{x_{k}\}\subset S^{1}\times c_{i_{k}}$. Thus
$c_{1}(w)$ evaluates to $1$ on each component
$\bar{R}_{k}\subset\bar{R}$. We may then consider the instanton Floer
homology group $I_{*}(Z|\bar{R})_{w}$.

\begin{lemma}\label{lem:relaxed-independence}
    Subject to the conditions we have imposed, the Floer homology
    group $I_{*}(Z|\bar{R})_{w}$ is independent of the choices made.
    In particular, $I_{*}(Z|\bar{R})_{w}$ is isomorphic to
    $\SHI(M,\gamma)$.
\end{lemma}

\begin{proof}
    By a sequence of applications of the excision property of Floer
    homology \cite{Floer-Durham-paper, KM-sutures}, we shall establish that
    $I_{*}(Z|\bar{R})_{w}$ is isomorphic to $I_{*}(Z'|\bar{R}')_{w'}$,
    where the latter arises from the same construction but with a
    \emph{connected} surface $T'$. Thus $I_{*}(Z'|\bar{R}')_{w'}$ is
    isomorphic to $\SHI(M,\gamma)$ by definition: its independence
    of the choices made is proved in \cite{KM-sutures}.

    We will show how to reduce the number of components
    of $T$ by one. Following the argument of \cite[section
    7.4]{KM-sutures}, we have an isomorphism
    \begin{equation}\label{eq:u-to-w}
                    I_{*}(Z|\bar{R})_{w} \cong I_{*}(Z|\bar{R})_{u},
    \end{equation}
    where $u\to Z$ is the complex line bundle whose first Chern class
    is dual to the cycle $c'\subset Z$. 
    We shall suppose in the fist instance that at least one of $c_{i}$
    or $c_{j}$ is non-separating in the corresponding component
    $T_{i}$ or $T_{j}$.
    Since $c_{1}(u)$ is odd on
    the $2$-tori $S^{1}\times c_{i}$ and $S^{1}\times c_{j}$, we can
    apply Floer's excision theorem (see also
    \cite[Theorem~7.7]{KM-sutures}): we cut $Z$ open along these two
    $2$-tori and glue back to obtain a new pair $(Z' | \bar{R}')$,
    carrying a line bundle $u'$, and we have
    \[
         I_{*}(Z|\bar{R})_{u} \cong  I_{*}(Z'|\bar{R}')_{u'}.            
    \]
    Reversing the construction that led to the isomorphism
    \eqref{eq:u-to-w}, we next have
    \[
         I_{*}(Z'|\bar{R}')_{u'} \cong  I_{*}(Z'|\bar{R}')_{w'},            
     \]
    where the line bundle
    $w'$ is dual to a collection of circles $S^{1}\times\{x'_{k'}\}$,
    one for each component of $\bar{R}'$. 
    The pair $(Z',\bar{R}')$ is obtained from the sutured
    manifold $(M,\gamma)$ by the same construction that led to
    $(Z,R)$, but with a surface $T'$ having one fewer components: the
    components $T_{i}$ and $T_{j}$ have been joined into one component
    by cutting open along the circles $c_{i}$ and $c_{j}$ and
    reglueing.

    If both $c_{i}$ and $c_{j}$ are separating in $T_{i}$ and $T_{j}$
    respectively, then the above argument fails, because $T'$ will
    have the same number of components as $T$. In this case, we can
    alter $T_{i}$ and $c_{i}$ to make a new $T'_{i}$ and $c'_{i}$,
    with $c'_{i}$ non-separating in $T'_{i}$. For example, we may
    replace $Z$ by the disjoint union $Z \amalg Z_{*}$, where $Z_{*}$
    is a product $S^{1}\times T_{*}$, with $T_{*}$ of genus $2$. In
    the same manner as above, we can cut $Z$ along $S^{1}\times c_{i}$
    and cut $Z_{*}$ along $S^{1}\times c_{*}$, and then reglue,
    interchanging the boundary components. The effect of this is to
    replace $T_{i}$ be a surface $T'_{i}$ of genus one larger. 
    We can take $c'_{i}$ to be a non-separating curve on $T_{*}
    \setminus c_{*}$.
\end{proof}

\subsection{Instanton homology for knots and links}
\label{subsec:inst-homology-link}

Consider a link $K$ in a closed oriented $3$-manifold $Y$. Following
Juh\'asz
\cite{Juhasz-1}, we can associate to $(Y,K)$ a sutured manifold
$(M,\gamma)$ by taking $M$ to be the link complement and taking the
sutures $s(\gamma)$ to consist of two oppositely-oriented meridional
curves on each of the tori in $\partial M$. As in \cite{KM-sutures},
where the case of knots was discussed, we take Juh\'asz'
prescription as a definition for the instanton knot (or link) homology
of the pair $(Y,K)$:

\begin{definition}[\textit{cf.} \cite{Juhasz-1}]
    We define the instanton homology of the link $K\subset Y$ to be
    the instanton Floer homology of the sutured manifold $(M,\gamma)$
    obtained from the link complement as above. Thus,
    \[
                    \KHI(Y,K) = \SHI(M,\gamma).
    \]
    \CloseDef
\end{definition}

Although we are free to choose any admissible closure $Z$ in
constructing $\SHI(M,\gamma)$, we can exploit the fact that we are
dealing with a link complement to narrow our choices. Let $r$ be the
number of components of the link $K$. Orient $K$ and choose a
longitudinal oriented curve $l_{i}\subset \partial M$ on the
peripheral torus of each component $K_{i}\subset K$. Let $F_{r}$ be a
genus-1 surface with $r$ boundary components, $\delta_{1},\dots,\delta_{r}$. Form a closed manifold
$Z$ by attaching $F_{r}\times S^{1}$ to $M$ along their
boundaries:
\begin{equation}\label{eq:special-closure}
                Z = (Y\setminus N^{o}(K))
                 \cup (F_{r}\times S^{1}).
\end{equation} 
The attaching is done so that the curve $p_{i}\times
S^{1}$ for $p_{i}\in \delta_{i}$ is attached to the meridian of
$K_{i}$ and $\delta_{i}\times \{q\}$ is attached to the chosen
longitude $l_{i}$. We can view $Z$ as a closure of $(M,\gamma)$ in
which the auxiliary surface $T$ consists of $r$ annuli,
\[
    T = T_{1}\cup \dots \cup T_{r}.
\]
The two
sutures of the product sutured manifold
$[-1,1]\times T_{i}$ are attached to meridional sutures on the
components of $\partial M$ corresponding to $K_{i}$ and $K_{i-1}$ in
some cyclic ordering of the components.
Viewed this way, the corresponding surface
$\bar{R}\subset Z$ is the torus
\[
            \bar{R} = \nu \times S^{1}
\]
where $\nu\subset F_{r}$ is a closed curve representing a generator
of the homology of the closed genus-1 surface obtained by adding disks
to $F_{r}$.
Because $\bar{R}$ is a torus,
the group
$I_{*}(Z|\bar{R})_{w}$ can be more simply described as the
generalized eigenspace of $\mu(y)$ belonging to the eigenvalue $2$,
for which we temporarily introduce the notation
$I_{*}(Z)_{w,+2}$. Thus we can write
\[
        \KHI(Y,K) = I_{*}(Z)_{w,+2}.
\]

An important special case for us is when $K \subset Y$ 
is null-homologous in $Y$ with its given orientation. In this case,
we may choose a Seifert surface $\Sigma$, which we regard as a
properly embedded oriented surface in $M$ with oriented boundary a
union of longitudinal curves, one for each component of $K$. When a
Seifert surface is given, we have a \emph{uniquely preferred} closure $Z$,
obtained as above but using the longitudes provided by $\partial
\Sigma$.   Let us fix a Seifert surface
$\Sigma$ and write $\sigma$ for its homology class in
$H_{2}(M,\partial M)$. The preferred closure of the sutured link
complement is entirely determined by $\sigma$.

\subsection{The decomposition into generalized eigenspaces}

We continue to suppose that $\Sigma$ is a Seifert surface for the
null-homologous oriented knot $K\subset Y$. We write $(M,\gamma)$ for
the sutured link complement and $Z$ for the preferred closure.

The homology class $\sigma = [\Sigma]$ in $H_{2}(M,\partial M)$
extends to a class $\bar\sigma = [\bar\Sigma]$ in $H_{2}(Z)$: the
surface $\bar\Sigma$ is formed from the Seifert surface $\Sigma$ and
 $F_{r}$,
\[
            \bar\Sigma = \Sigma\cup  F_{r}.
\]
The homology class $\bar\sigma$ determines an endomorphism
\[
                \mu(\bar\sigma) : I_{*}(Z)_{w,+2} \to  I_{*}(Z)_{w,+2}.
\]
This endomorphism is traceless, a consequence of the relative $\Z/8$
grading:  there is an endomorphism $\epsilon$ of $I_{*}(Z)_{w}$ given
by multiplication by $(\sqrt{-1})^{s}$ on the part of relative grading
$s$, and this $\epsilon$ commutes with $\mu(y)$ and anti-commutes
with $\mu(\bar\sigma)$. We write this traceless
endomorphism as
\begin{equation}\label{eq:mu-o}
                \mu^{o}(\sigma) \in \sl( \KHI(Y,K)).
\end{equation}
Our notation hides the fact that the construction depends (a priori)
on the existence of the preferred closure $Z$, so that $\KHI(Y,K)$ can
be canonically identified with $I_{*}(Z)_{w,+2}$.

It now follows from \cite[Proposition~7.5]{KM-sutures} that the eigenvalues of
$\mu^{o}(\sigma)$ are even integers $2j$ in the range
$-2\bar{g}+2 \le 2j \le 2\bar{g}-2$, where $\bar{g}=g(\Sigma)+r $ is
the genus of $\bar{\Sigma}$. Thus:

\begin{definition}
    For a null-homologous oriented link $K\subset Y$ with a chosen
    Seifert surface $\Sigma$, we write
    \[
                    \KHI(Y,K,[\Sigma], j)
                    \subset \KHI(Y,K)
    \]
    for the generalized eigenspace of $\mu^{o}([\Sigma])$ belonging to
    the eigenvalue $2j$, so that
   \[
                    \KHI(Y,K) =
                    \bigoplus_{j=-g(\Sigma)+1-r}^{g(\Sigma)-1+r}
                    \KHI(Y,K,[\Sigma], j),
   \]
   where $r$ is the number of components of $K$.
    If $Y$ is a homology sphere, we may omit
    $[\Sigma]$ from the notation; and if $Y$ is $S^{3}$ then we simply
    write $\KHI(K,j)$. \CloseDef
\end{definition}

\begin{remark}
    The authors believe that, for a general sutured manifold
    $(M,\gamma)$, one can define a unique linear map
\[
       \mu^{o} : H_{2}(M,\partial M) \to \sl ( \SHI(M,\gamma))   
\]
characterized by the property that for any admissible closure
        $(Z,\bar{R})$ and any $\bar{\sigma}$ in $H_{2}(Z)$ 
        extending $\sigma \in H_{2}(M,\partial M)$ we have
       \[ \mu^{o}(\sigma) = \text{traceless part of $\mu(\bar\sigma)$},
        \]
under a suitable
identification of $I_{*}(Z| \bar{R})_{w}$ with $\SHI(M,\gamma)$. The
authors will return to this question in a future paper. For now, we
are exploiting the existence of a preferred closure $Z$ so as to
side-step the issue of whether $\mu^{o}$ would be well-defined,
independent of the choices made.
\end{remark}

\subsection{The mod 2 grading}
\label{subsec:mod-2-grading}

If $Y$ is a closed $3$-manifold, then the instanton homology group
$I_{*}(Y)_{w}$ has a canonical decomposition into parts of even and
odd grading mod $2$. For the purposes of this paper, we normalize our
conventions so that the two generators of $I_{*}(T^{3})_{w}=\C^{2}$
are in \emph{odd} degree. As in \cite[section~25.4 ]{KM-book}, the
canonical mod $2$ grading is then essentially determined by the
property that, for a cobordism $W$ from a manifold $Y_{-}$ to $Y_{+}$,
the induced map on Floer homology has even or odd grading according to
the parity of the integer
\begin{equation}\label{eq:iota-W} \iota(W) = \frac{1}{2}
     \Bigl( \chi(W) + \sigma(W) + 
       b_1(Y_+) - b_0(Y_+) - b_1(Y_-) + b_0(Y_-)\Bigr).
\end{equation}
(In the case of connected manifolds $Y_{+}$ and $Y_{-}$, 
this formula reduces to the one that appears in \cite{KM-book} for the monopole
case. There is more than one way to extend the formula to the case of
disconnected manifolds, and we have simply chosen one.) 
By declaring that the generators for
$T^{3}$ are in odd degree, we ensure that the canonical mod $2$
gradings behave as expected for disjoint unions of the $3$-manifolds.
Thus, if $Y_{1}$ and $Y_{2}$ are the connected components of a
$3$-manifold $Y$ and $\alpha_{1}\otimes \alpha_{2}$ is a class on $Y$
obtained from $\alpha_{i}$ on $Y_{i}$, then $\gr(\alpha_{1}\otimes
\alpha_{2})$ is $\gr(\alpha_{1}) + \gr(\alpha_{2})$ in $\Z/2$ as
expected. 

Since the Floer homology $\SHI(M,\gamma)$ of a sutured manifold
$(M,\gamma)$ is defined in terms of $I_{*}(Z)_{w}$ for an admissible
closure $Z$, it is tempting to try to define a canonical mod $2$
grading on $\SHI(M,\gamma)$ by carrying over the canonical mod $2$
grading from $Z$. This does not work, however, because the result will
depend on the choice of closure. This is illustrated by the fact that
the mapping torus of a Dehn twist on $T^{2}$ may have Floer homology
in \emph{even} degree in the canonical mod $2$ grading  (depending on
the sign of the Dehn twist), despite the fact that both $T^{3}$ and
this mapping torus can be viewed as closures of the same sutured
manifold.

We conclude from this that, without auxiliary choices, there is no
\emph{canonical} mod $2$ grading on $\SHI(M,\gamma)$ in general: only
a relative grading. Nevertheless, in the special case of an oriented
null-homologous knot or link $K$ in a closed $3$-manifold $Y$, we
\emph{can} fix a convention that gives an absolute mod $2$ grading,
once a Seifert surface $\Sigma$ for $K$ is given.  We simply take the
preferred closure $Z$ described above in
section~\ref{subsec:inst-homology-link}, using $\partial\Sigma$ again
to define the longitudes, so that $\KHI(Y,K)$ is identified with
$I_{*}(Z)_{w,+2}$, and we use the canonical mod $2$ grading from the
latter.

With this convention, the unknot $U$ has $\KHI(U)$ of rank $1$, with
a single generator in odd grading mod $2$.

\section{The skein sequence}
\label{sec:skein}

\subsection{The long exact sequence}

Let $Y$ be any closed, oriented $3$-manifold, and let $K_{+}$, $K_{-}$
and $K_{0}$ be any three oriented knots or links in $Y$ which are related by
the standard skein moves: that is, all three links coincide outside a
ball $B$ in
$Y$, while inside the ball they are as shown in
Figure~\ref{fig:Oriented-Skein}.  There are two cases which occur
here: the two strands of $K_{+}$ in $B$ may belong to the same
component of the link, or to different components. In the first case
$K_{0}$ has one more component than $K_{+}$ or $K_{-}$, while in the
second case it has one fewer.

\begin{theorem}[\textit{cf.} Theorem 5 of \cite{Floer-Durham-paper}]
\label{thm:skein}
    Let $K_{+}$, $K_{-}$ and $K_{0}$ be oriented links in $Y$ as
    above. Then, in the case that $K_{0}$ has one more component than
    $K_{+}$ and $K_{-}$, there is a long exact sequence relating the
    instanton homology groups of the three links,
    \begin{equation}\label{eq:skein-first}
               \cdots\to \KHI(Y,K_{+}) \to \KHI(Y,K_{-}) \to \KHI(Y,K_{0}) \to
                \cdots.
    \end{equation}
    In the case that $K_{0}$ has fewer components that $K_{+}$ and
    $K_{-}$, there is a long exact sequence
    \begin{equation}\label{eq:skein-second}
               \cdots\to \KHI(Y,K_{+}) \to \KHI(Y,K_{-}) \to
               \KHI(Y,K_{0})\otimes V^{\otimes 2} \to
                \cdots
    \end{equation}
    where $V$ a 2-dimensional vector space arising as
    the instanton Floer homology of the sutured manifold
    $(M,\gamma)$, with $M$ the solid torus $S^{1}\times D^{2}$
    carrying four parallel sutures $S^{1}\times \{p_{i}\}$ for four
    points $p_{i}$ on $\partial D^{2}$ carrying alternating
    orientations.
\end{theorem}

\begin{proof}
    Let $\lambda$ be a standard circle in the complement of $K_{+}$
    which encircles the two strands of $K_{+}$ with total linking
    number zero, as shown in Figure~\ref{fig:K-plus-with-lambda}.
\begin{figure}
    \begin{center}
        \includegraphics[scale=0.6]{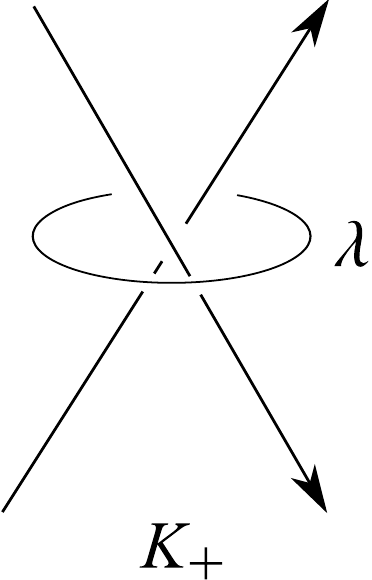}
    \end{center}
    \caption{\label{fig:K-plus-with-lambda}
    The knot $K_{+}$, with a standard circle $\lambda$ around a
    crossing, with linking number zero.}
\end{figure}
    Let
    $Y_{-}$ and $Y_{0}$ be the $3$-manifolds obtained from $Y$ by
    $-1$-surgery and $0$-surgery on $\lambda$ respectively. Since
    $\lambda$ is disjoint from $K_{+}$, a copy of $K_{+}$ lies in
    each, and we have new pairs $(Y_{-1},K_{+})$ and $(Y_{0},K_{+})$.
    The pair $(Y_{-1},K_{+})$ can be identified with $(Y,K_{-})$.
\begin{figure}
    \begin{center}
        \includegraphics[scale=0.6]{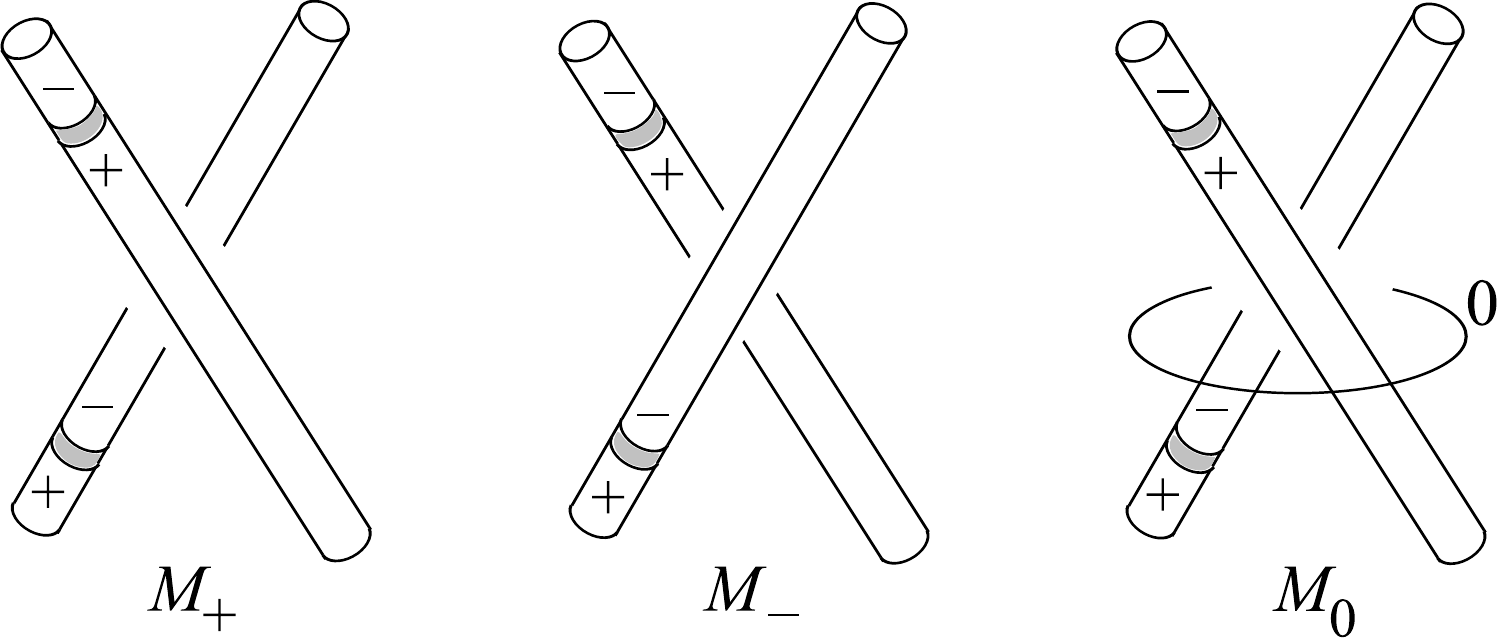}
    \end{center}
    \caption{\label{fig:Skein-Tubes}
    Sutured manifolds obtained from the knot complement, related by a
    surgery exact triangle.}
\end{figure}

    Let
    $(M_{+},\gamma_{+})$, $(M_{-},\gamma_{-})$ and
    $(M_{0},\gamma_{0})$ be the sutured manifolds associated to the
    links $(Y,K_{+})$, $(Y,K_{-})$ and $(Y_{0},K_{0})$ respectively:
    that is, $M_{+}$, $M_{-}$ and $M_{0}$ are the link complements of
    $K_{+}\subset Y$, $K_{-}\subset Y$ and $K_{0}\subset Y_{0}$
    respectively, and there are two sutures on each boundary
    component. (See Figure~\ref{fig:Skein-Tubes}.)
    The sutured manifolds $(M_{-},\gamma_{-})$ and
    $(M_{0}, \gamma_{0})$ are obtained from $(M_{+},\gamma_{+})$ by
    $-1$-surgery and $0$-surgery respectively on the circle
    $\lambda\subset M_{+}$. If $(Z,\bar{R})$ is any admissible closure
    of $(M_{+},\gamma_{+})$ then surgery on $\lambda\subset Z$ yields
    admissible closures for the other two sutured manifolds. From
    Floer's surgery exact triangle \cite{Braam-Donaldson}, it follows
    that there is a long exact sequence
    \begin{equation}\label{eq:SHI-long-exact}
               \cdots\to \SHI(M_{+},\gamma_{+}) \to
               \SHI(M_{-},\gamma_{-}) \to
               \SHI(M_{0},\gamma_{0}) \to
                \cdots
    \end{equation}
     in which the maps are induced by surgery cobordisms between
     admissible closures of the sutured manifolds.

     By definition, we have
     \[
\begin{aligned}
    \SHI(M_{+},\gamma_{+}) &= \KHI(Y,K_{+}) \\
    \SHI(M_{-},\gamma_{-}) &= \KHI(Y,K_{-}) .
\end{aligned}
     \]
     \begin{figure}
    \begin{center}
        \includegraphics[scale=0.6]{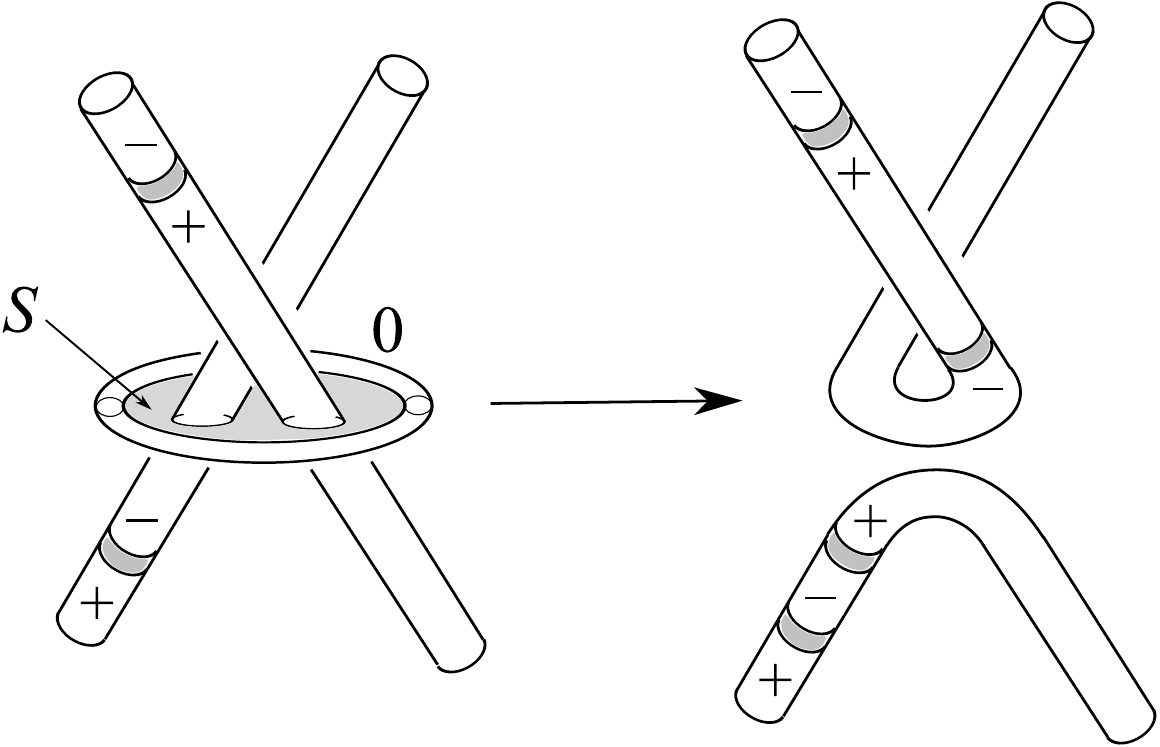}
    \end{center}
    \caption{\label{fig:Decompose-M0}
    Decomposing $M_{0}$ along a product annulus to obtain a link
    complement in $S^{3}$.}
\end{figure}%
    However, the situation for $(M_{0}, \gamma_{0})$ is a little
    different. The manifold $M_{0}$ is obtained by zero-surgery on the
    circle $\lambda$ in $M_{+}$, as indicated in
    Figure~\ref{fig:Skein-Tubes}. This sutured manifold contains a
    product annulus $S$, consisting of the union of the
    twice-punctured disk shown in Figure~\ref{fig:Decompose-M0} and
    a disk $D^{2}$ in the surgery solid-torus $S^{1}\times D^{2}$. As
    shown in the figure, sutured-manifold decomposition along the
    annulus $S$ gives a sutured manifold $(M'_{0},\gamma'_{0})$ in
    which $M'_{0}$ is the link complement of $K_{0}\subset Y$:
    \[
                (M_{0},\gamma_{0}) \decomp{S} (M'_{0}, \gamma'_{0}).
    \]
    By Proposition~6.7 of \cite{KM-sutures} (as adapted to the
    instanton homology setting in section~7.5 of that paper), we
    therefore have an isomorphism
    \[
                 \SHI (M_{0},\gamma_{0})\cong \SHI (M'_{0},
                 \gamma'_{0}).
    \]

    We now have to separate cases according to the number of
    components of $K_{+}$ and $K_{0}$. If the two strands of $K_{+}$
    at the crossing belong to the same component, then every component
    of $\partial M'_{0}$ contains exactly two, oppositely-oriented
    sutures, and we therefore have
    \[
                    \SHI (M'_{0},
                 \gamma'_{0}) = \KHI(Y, K_{0}).
    \]
    In this case, the sequence \eqref{eq:SHI-long-exact} becomes the
    sequence in the first case of the theorem.

    \begin{figure}
    \begin{center}
        \includegraphics[scale=0.7]{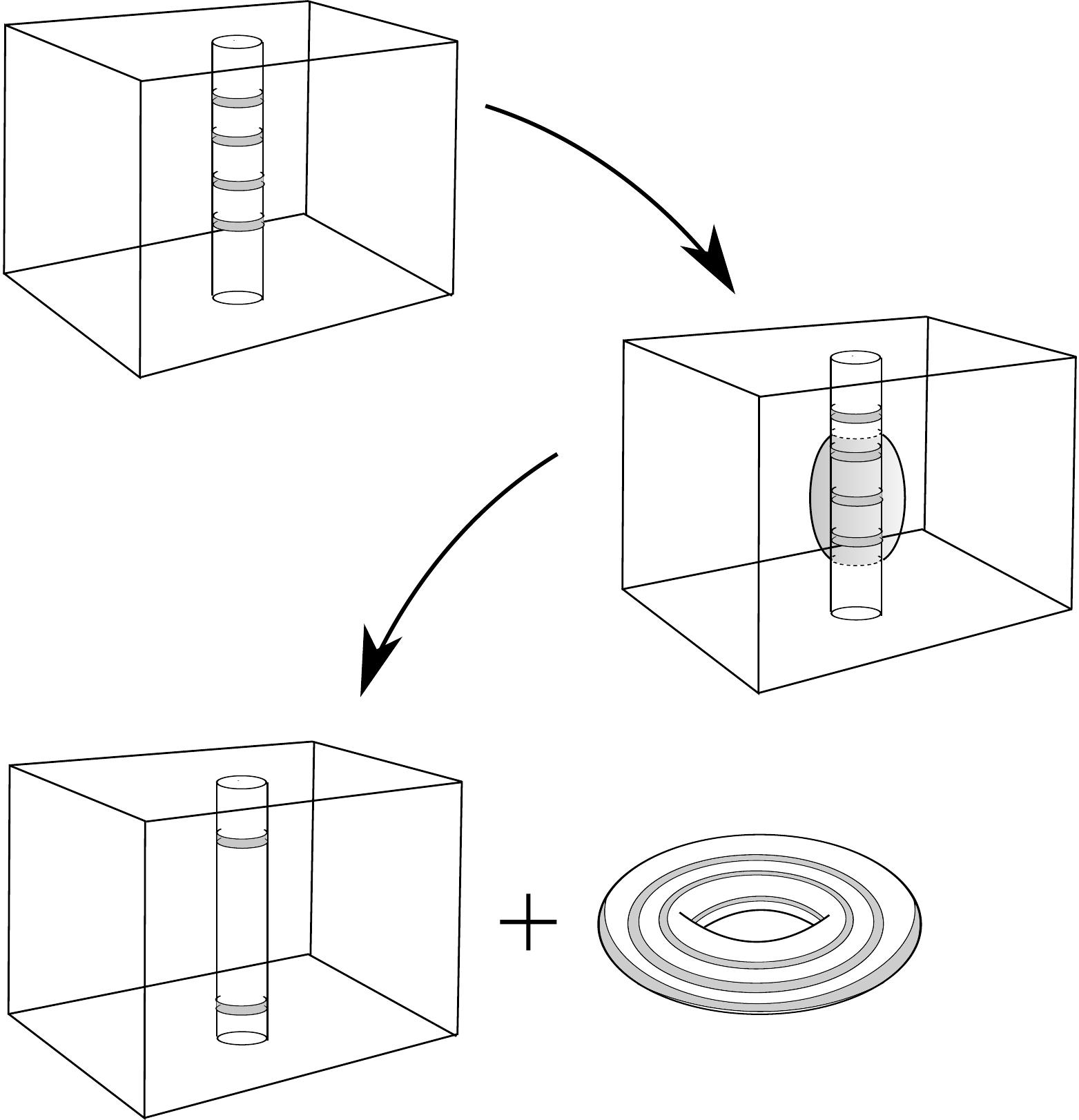}
    \end{center}
    \caption{\label{fig:Remove-Sutures}
    Removing some extra sutures using a decomposition along a product
    annulus. The solid torus in the last step has four
    sutures.}
\end{figure}

    If the two strands of $K_{+}$ belong to different components, then
    the corresponding boundary components of $M_{+}$ each carry two
    sutures. These two boundary components become one boundary
    component in $M'_{0}$, and the decomposition along $S$ introduces
    two new sutures; so the resulting boundary component in $M'_{0}$
    carries six meridional sutures, with alternating orientations.
    Thus $(M'_{0}, \gamma'_{0})$ fails to be the sutured manifold
    associated to the link $K_{0}\subset Y$, on account of having four
    additional sutures. As shown in Figure~\ref{fig:Remove-Sutures}
    however, the number of  sutures on a torus boundary component can
    always be reduced by $2$ (as long as there are at least four to
    start with) by using a decomposition along a  separating annulus.
    This decomposition results in a manifold with one additional
    connected component, which is a solid torus with four longitudinal
    sutures. This operation needs to be performed twice to reduce the
    number of sutures in $M'_{0}$ by four, so we obtain two copies of
    this solid torus. Denoting by $V$ the Floer homology of this
    four-sutured solid-torus, we therefore have
    \[
                 \SHI (M'_{0},
                 \gamma'_{0}) = \KHI(Y, K_{0})\otimes V\otimes V
    \]
    in this case. Thus the sequence \eqref{eq:SHI-long-exact} becomes
    the second long exact sequence in the theorem.

    At this point, all that remains is to show that $V$ is
    $2$-dimensional, as asserted in the theorem. We will do this
    indirectly, by identifying $V\otimes V$ as a $4$-dimensional
    vector space. Let $(M_{4},\gamma_{4})$ be the sutured solid-torus
    with $4$ longitudinal sutures, as described above, so that
    $\SHI(M_{4},\gamma_{4})=V$. Let $(M,\gamma)$ be two disjoint
    copies of $(M_{4},\gamma_{4})$,
    so that
    \[
                    \SHI(M,\gamma) = V\otimes V.
    \]

    We can describe an admissible closure of $(M,\gamma)$ (with a
    disconnected $T$ as in section~\ref{subsec:disconnected-T}) by
    taking $T$ to be four annuli: we attach $[-1,1]\times T$ to
    $(M,\gamma)$ to form $\bar{M}$ so that $\bar{M}$ is $\Sigma\times
    S^{1}$ with $\Sigma$ a four-punctured sphere. Thus
    $\partial\bar{M}$ consists of four tori, two of which belong to
    $\bar{R}_{+}$ and two to $\bar{R}_{-}$. The closure $(Y,\bar{R})$
    is obtained by gluing the tori in pairs; and this can be done so
    that $Y$ has the form $\Sigma_{2}\times S^{1}$, where $\Sigma_{2}$
    is now a closed surface of genus $2$. The surface $\bar{R}$ in
    $\Sigma_{2}\times S^{1}$ has the form $\gamma\times S^{1}$, where
    $\gamma$ is a union of two disjoint closed curves in independent
    homology classes. The line bundle $w$ has $c_{1}(w)$ dual to
    $\gamma'$, where $\gamma'$ is a curve on $\Sigma_{2}$ dual to one
    component of $\gamma$.

    Thus we can identify $V\otimes V$ with the generalized eigenspace
    of $\mu(y)$
    belonging to the eigenvalue $+2$ in the Floer homology
    $I_{*}(\Sigma_{2}\times S^{1})_{w}$, 
    \begin{equation}
        \label{eq:VVisSigma2}
        V\otimes V = I_{*}(\Sigma_{2} \times S^{1})_{w,+2},
    \end{equation}
where $w$ is dual to a curve
    lying on $\Sigma_{2}$. Our next task is therefore to identify this
    Floer homology group. This was done (in slightly different
    language) by Braam and Donaldson
    \cite[Proposition~1.15]{Braam-Donaldson}.  The
    main point is to identify the relevant representation variety in
    $\bonf(Y)_{w}$, for which we quote:

    \begin{lemma}[{\cite{Braam-Donaldson}}]
   \label{lem:Sigma2-calc}
        For $Y=\Sigma_{2}\times S^{1}$ and $w$ as above,
        the critical-point set of the Chern-Simons functional in
        $\bonf(Y)_{w}$ consists of two disjoint $2$-tori. Furthermore,
        the Chern-Simons functional is of Morse-Bott type along its
        critical locus. \qed
    \end{lemma}

        To continue the calculation, following
        \cite{Braam-Donaldson}, it now follows from the
        lemma that $I_{*}(\Sigma_{2}\times S^{1})_{w}$ has dimension at most
        $8$ and that the even and odd parts of this Floer group, with
        respect to the relative mod 2 grading, have equal dimension:      
        each at most $4$. On the other hand, the group
        $I_{*}(\Sigma_{2}\times S^{1}| \Sigma_{2})_{w}$ is non-zero.
        So the generalized eigenspaces belonging to the
        eigenvalue-pairs $((-1)^{r}2, i^{r}2)$, for $r=0,1,2,3$, are
        all non-zero. Indeed, each of these generalized eigenspaces is
        $1$-dimensional, by Proposition~7.9 of \cite{KM-sutures}.
        These four 1-dimensional generalized eigenspaces all belong
        to the same relative mod-2 grading. It follows that
        $I_{*}(\Sigma_{2}\times S^{1})_{w}$  is 8-dimensional, and can
        be identified as a vector space with the homology of the
        critical-point set.  The generalized eigenspace belonging to
        $+2$ for the operator $\mu(y)$ is therefore $4$-dimensional;
        and this is $V\otimes V$. This completes the argument.
    \end{proof}

\subsection{Tracking the mod 2 grading}

Because we wish to examine the Euler characteristics, we need to know
how the canonical mod 2 grading behaves under the maps in
Theorem~\ref{thm:skein}. This is the content of the next lemma.

\begin{lemma}\label{lem:mod-2-sequence}
    In the situation of Theorem~\ref{thm:skein}, suppose that the link
    $K_{+}$ is null-homologous (so that $K_{-}$ and $K_{0}$ are
    null-homologous also). Let $\Sigma_{+}$ be a Seifert surface for
    $K_{+}$, and let $\Sigma_{-}$ and $\Sigma_{0}$ be Seifert surfaces
    for the other two links, obtained from $\Sigma_{+}$ by a
    modification in the neighborhood of the crossing. Equip the
    instanton knot homology groups of these links with their canonical
    mod $2$ gradings, as determined by the preferred closures arising
    from these Seifert surfaces.  
    Then in the first case of the two cases of the
    theorem, the map from  $\KHI(Y,K_{-})$ to $\KHI(Y,K_{0})$ in the
    sequence \eqref{eq:skein-first} has odd degree, while the other
    two maps have even degree, with respect to the canonical mod 2
    grading.

    In the second case, if we grade the 4-dimensional vector space
    $V\otimes V$ by identifying it with $I_{*}(\Sigma_{2}\times
    S^{1})_{w,+2}$ as in \eqref{eq:VVisSigma2}, then the map from
      $\KHI(Y,K_{0})\otimes V^{\otimes 2}$ to $\KHI(Y,K_{+})$
     in \eqref{eq:skein-second}
    has odd degree,  while the other
    two maps have even degree.
\end{lemma}

\begin{proof}
    We begin with the first case.  Let $Z_{+}$ be the preferred
    closure of the sutured knot complement $(M_{+},\gamma_{+})$
    obtained from the knot $K_{+}$, as defined by
    \eqref{eq:special-closure}.  In the notation of the proof of
    Theorem~\ref{thm:skein}, the curve $\lambda$ lies in $Z_{+}$. Let
    us write $Z_{-}$ and $Z_{0}$ for the manifolds obtained from
    $Z_{+}$ by $-1$-surgery and $0$-surgery on $\lambda$ respectively.
    It is a straightforward observation that $Z_{-}$ and $Z_{0}$ are
    respectively the preferred closures of the sutured complements of
    the links $K_{-}$ and $K_{0}$.  The surgery cobordism $W$ from
    $Z_{+}$ to $Z_{-}$ gives rise to the map from $\KHI(Y,K_{+})$ to
    $\KHI(Y,K_{-})$.  This $W$ has the same homology as the cylinder
    $[-1,1]\times Z_{+}$ blown up at a single point. The quantity
    $\iota(W)$ in \eqref{eq:iota-W} is therefore even, and it follows
    that the map
    \[
              \KHI(Y,K_{+}) \to \KHI(Y,K_{-})      
    \]
    has even degree. The surgery cobordism $W_{0}$ induces a map
    \begin{equation}\label{eq:second-cobordism}
                    I_{*}(Z_{-})_{w} \to I_{*}(Z_{0})_{w}
    \end{equation}
    which has odd degree, by another application of \eqref{eq:iota-W}.
    This concludes the proof of the first case.
     
    In the second case of the theorem,
    we still have a long exact
    sequence
    \[
          \to I_{*}(Z_{+})_{w}  \to     I_{*}(Z_{-})_{w} \to
          I_{*}(Z_{0})_{w}     \to
     \]
     in which the map $I_{*}(Z_{-})_{w} \to I_{*}(Z_{0})_{w}$ is
     odd and the other two are even. 
    However, it is no longer true that the manifold $Z_{0}$ is
     the preferred closure of the sutured manifold obtained from
     $K_{0}$. The manifold $Z_{0}$ can be described as being obtained
     from the complement of $K_{0}$ by attaching $G_{r}\times S^{1}$,
     where $G_{r}$ is a surface of genus $2$ with $r$ boundary
     components. Here $r$ is the number of components of $K_{0}$, and
     the attaching is done as before, so that the curves 
      $\partial G_{r}\times
     \{q\}$ is attached to the longitudes and the curves 
     $\{p_{i}\}\times S^{1}$
     are attached to the meridians. The \emph{preferred} closure, on
     the other hand, is defined using a surface $F_{r}$ of genus
     $1$, not genus $2$. We write $Z'_{0}$ for the preferred closure, and our
     remaining task is to compare the instanton Floer homologies of
     $Z_{0}$ and $Z'_{0}$, with their canonical $\Z/2$ gradings.

     An application of Floer's excision theorem provides an
     isomorphism
     \[
                 I_{*}(Z_{0})_{w,+2} \to I_{*}(Z'_{0})_{w,+2} \otimes
                 I_{*}(\Sigma_{2}\times S^{1})_{w,+2}
      \]
     where (as before) the class $w$ in the last term is dual to a
     non-separating curve in the genus-2 surface $\Sigma_{2}$. 
    \begin{figure}
    \begin{center}
        \includegraphics[scale=0.4]{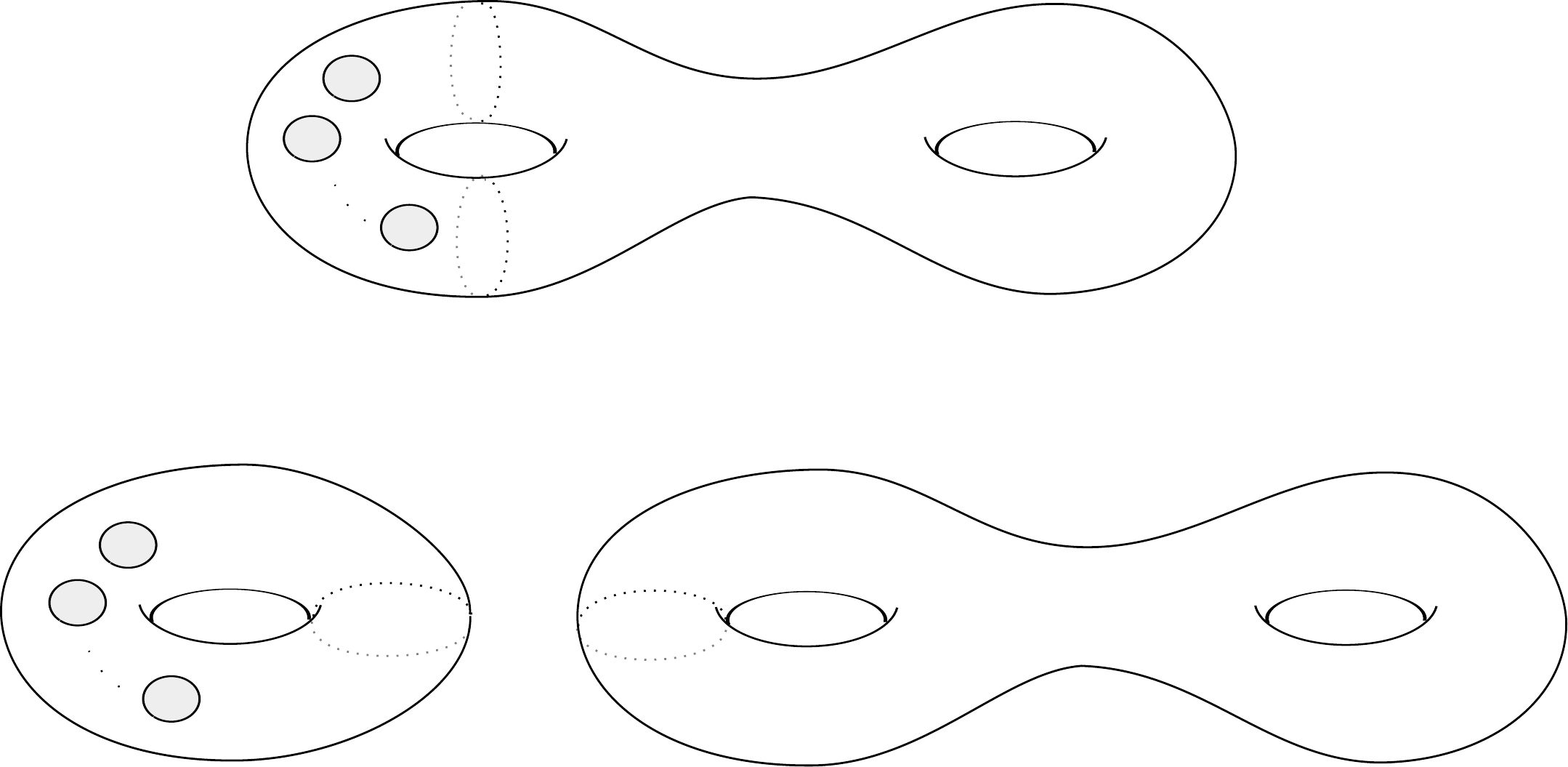}
    \end{center}
    \caption{\label{fig:F-and-G}
    The surfaces $G_{r}$ and $F_{r} \amalg \Sigma_{2}$, used in constructing $Z_{0}$
    and $Z'_{0}$ respectively.}
\end{figure}
    (See
     Figure~\ref{fig:F-and-G} which depicts the excision cobordism
     from $G_{r}\times S^{1}$ to $(F_{r}\amalg \Sigma_{2})\times
     S^{1}$, with the $S^{1}$ factor suppressed.) The
     isomorphism is realized by an explicit cobordism $W$, with
     $\iota(W)$ odd, which accounts for the difference between the
     first and second cases and concludes the proof.
\end{proof}

\subsection{Tracking the eigenspace decomposition}

The next lemma is similar in spirit to Lemma~\ref{lem:mod-2-sequence},
but deals with eigenspace decomposition rather than the mod $2$
grading.

\begin{lemma}\label{lem:eigenspace-sequence}
    In the situation of Theorem~\ref{thm:skein}, suppose again that
    the links
    $K_{+}$, $K_{-}$ and $K_{0}$ are
    null-homologous. Let $\Sigma_{+}$ be a Seifert surface for
    $K_{+}$, and let $\Sigma_{-}$ and $\Sigma_{0}$ be Seifert surfaces
    for the other two links, obtained from $\Sigma_{+}$ by a
    modification in the neighborhood of the crossing.
    Then in the first case of the two cases of the theorem, the
    maps in the long exact
    sequence \eqref{eq:skein-first} intertwine the three operators
    $\mu^{o}([\Sigma_{+}])$, $\mu^{o}([\Sigma_{-}])$ and
    $\mu^{o}([\Sigma_{0}])$. In particular then, we have a long exact
    sequence
    \begin{equation*}
               \to \KHI(Y,K_{+},[\Sigma_{+}],j) \to
               \KHI(Y,K_{-},[\Sigma_{-}],j) \to
               \KHI(Y,K_{0},[\Sigma_{0}],j) \to
    \end{equation*}
    for every $j$.

    In the second case of Theorem~\ref{thm:skein}, the maps in the
    long exact sequence \eqref{eq:skein-second} intertwine the
    operators $\mu^{o}([\Sigma_{+}])$ and $\mu^{o}([\Sigma_{-}])$ on
    the first two terms with the operator
    \[
                        \mu^{o}([\Sigma_{0}]) \otimes 1 +
                       1 \otimes  \mu([\Sigma_{2}])
    \]
    acting on
    \[
                \KHI(Y,K_{0})\otimes I_{*}(\Sigma_{2}\times
                S^{1})_{w,+2}\cong \KHI(Y,K_{0})\otimes V^{\otimes 2}.
    \]
\end{lemma}

\begin{proof}
    The operator $\mu^{o}([\Sigma])$ on the knot homology groups is
    defined in terms of the action of $\mu([\bar\Sigma])$ for a
    corresponding closed surface $\bar\Sigma$ in the preferred closure
    of the link complement. The maps in the long exact sequences arise
    from cobordisms between the preferred closures. The lemma follows
    from the fact that the corresponding closed surfaces are
    homologous in these cobordisms.
\end{proof}

\subsection{Proof of the main theorem}

For a null-homologous link $K \subset Y$ with a chosen Seifert surface
$\Sigma$, let us write
\[
\begin{aligned}
\chi (Y,K,[\Sigma])&= \sum_{j}
            \chi(\KHI(Y,K,[\Sigma],j))t^{j} \\
             &= \sum_{j} \bigl ( \dim\KHI_{0}(Y,K,[\Sigma],j) -
             \dim\KHI_{1}(Y,K,[\Sigma],j)\bigr) t^{j} \\
            &= \str( t^{\mu^{o}(\Sigma)/2}),
\end{aligned}
\]
where $\str$ denotes the alternating trace.
If $K_{+}$, $K_{-}$ and $K_{0}$ are three skein-related links with
corresponding Seifert surfaces $\Sigma_{+}$, $\Sigma_{-}$ and
$\Sigma_{0}$, then Theorem~\ref{thm:skein}, Lemma~\ref{lem:mod-2-sequence} and
Lemma~\ref{lem:eigenspace-sequence} together tell us that we have the
relation
\[
            \chi (Y,K_{+},[\Sigma_{+}]) -  \chi (Y,K_{-},[\Sigma_{-}]) +  \chi
            (Y,K_{0},[\Sigma_{0}]) = 0 
\]
in the first case of Theorem~\ref{thm:skein}, and
\[
            \chi (Y,K_{+},[\Sigma_{+}]) -  \chi (Y,K_{-},[\Sigma_{-}]) -  \chi
            (Y,K_{0},[\Sigma_{0}]) r(t) = 0 
\]
in the second case. Here $r(t)$ is the contribution from the term
$I_{*}(\Sigma_{2}\times S^{1})_{w,+2}$, so that
\[
                r(t) = \str (t^{\mu([\Sigma_{2}])/2}).
\]
From the proof of Lemma~\ref{lem:Sigma2-calc} we can read off the
eigenvalues of $[\Sigma_{2}]/2$: they are $1$, $0$ and $-1$, and the
$\pm 1$ eigenspaces are each $1$-dimensional. Thus
\[
               r(t) = \pm (t - 2 + t^{-1}).
\]

To determine the sign of $r(t)$, we need to know the canonical $\Z/2$
grading of (say) the $0$-eigenspace of $\mu([\Sigma_{2}])$ in
$I_{*}(\Sigma_{2}\times S^{1})_{w,+2}$. The trivial $3$-dimensional
cobordism from $T^{2}$ to $T^{2}$ can be decomposed as $N^{+}\cup
N^{-}$, where $N^{+}$ is a cobordism from $T^{2}$ to $\Sigma_{2}$ and
$N_{-}$ is a cobordism the other way. The $4$-dimensional cobordisms
$W^{\pm}= N^{\pm}\times S^{1}$ induce isomorphisms on the
$0$-eigenspace  of $\mu([T^{2}])=\mu([\Sigma_{2}])$; and
$\iota(W^{\pm})$ is odd. Since the generator for $T^{3}$ is in odd
degree, we conclude that the $0$-eigenspace of $\mu([\Sigma_{2}])$ is
in even degree, and that
\[
\begin{aligned}
    r(t) &= - (t - 2 + t^{-1}) \\
         &= - q(t)^{2}
\end{aligned}
\]
where
\[
         q(t)  = (t^{1/2}-t^{-1/2}).
\]

We can roll the two case of Theorem~\ref{thm:skein} into one by
defining the ``normalized'' Euler characteristic as
\begin{equation}\label{eq:renormalized}
                \tilde\chi(Y,K,[\Sigma]) =
                    q(t)^{1-r}\chi(Y,K,[\Sigma])
\end{equation}
where $r$ is the number of components of the link $K$. With this
notation we have:

\begin{proposition}
    For null-homologous skein-related links $K_{+}$, $K_{-}$ and
    $K_{0}$ with corresponding Seifert surface $\Sigma_{+}$,
    $\Sigma_{-}$ and $\Sigma_{0}$, the normalized Euler
    characteristics \eqref{eq:renormalized} satisfy
    \[
                 \tilde \chi (Y,K_{+},[\Sigma_{+}]) - \tilde \chi
            (Y,K_{-},[\Sigma_{-}])=   (t^{1/2}-t^{-1/2})\,\tilde \chi
            (Y,K_{0},[\Sigma_{0}]).    
    \]
    \qed
\end{proposition}

In the case of classical knots and links, we may write this simply as
\[
                 \tilde \chi (K_{+}) -  \tilde\chi
            (K_{-})=   (t^{1/2}-t^{-1/2})\,\tilde\chi
            (K_{0}).    
\]
This is the exactly the skein relation of the (single-variable)
normalized Alexander  polynomial
$\Delta$. The latter is
normalized so that  $\Delta=1$ for the unknot, whereas our $\tilde\chi$ is
$-1$ for the unknot because the generator of its knot homology is in
odd degree. We therefore have:

\begin{theorem}
For any link $K$ in $S^{3}$, we have
\[
                \tilde\chi(K) = - \Delta_{K}(t),
\]
where $\tilde\chi(K)$ is the normalized Euler characteristic
\eqref{eq:renormalized} and $\Delta_{K}$ is the Alexander polynomial
of the link with Conway's normalization.\qed
\end{theorem}

In the case that $K$ is a knot, we have $\tilde\chi(K)=\chi(K)$, which
is the case given in Theorem~\ref{thm:main} in the introduction. \qed

\begin{remark}
    The equality $r(t)= - q(t)^{2}$ can be interpreted as arising from
    the isomorphism
    \[
                  I_{*}(\Sigma_{2} \times S^{1})_{w,+2} \cong V\otimes V,
     \]
     with the additional observation that the isomorphism between
     these two is odd with respect to the preferred $\Z/2$ gradings.
\end{remark}

\section{Applications}
\label{sec:applications}

\subsection{Fibered knots}

In \cite{KM-sutures}, the authors adapted the argument of Ni \cite{Ni-A}
to establish a criterion for a knot $K$ in $S^3$ to be a fibered knot: in
particular, Corollary~7.19 of \cite{KM-sutures} states that $K$
is fibered if the following three conditions hold:
\begin{enumerate}
    \item the Alexander polynomial $\Delta_{K}(T)$ is monic, in the
    sense that its leading coefficient is $\pm 1$;
    \item the leading coefficient occurs in degree $g$, where
    $g$ is the genus of the knot; and
    \item the dimension of $\KHI(K,g)$ is $1$.
\end{enumerate}
It follows from our Theorem~\ref{thm:main} that the last of these
three conditions implies the other two. So we have:

\begin{proposition}\label{prop:fibered-knot}
    If $K$ is a knot in $S^{3}$ of genus $g$, then $K$ is fibered if
    and only if the dimension of $\KHI(K,g)$ is $1$. \qed
\end{proposition}

\subsection{Counting representations}

We describe some applications to representation varieties associated
to 
classical knots $K\subset S^{3}$. The
instanton knot homology $\KHI(K)$ is defined in terms of the preferred
closure $Z=Z(K)$ described at \eqref{eq:special-closure}, and
therefore involves the flat connections
\[
                \Rep(Z)_{w} \subset \bonf(Z)_{w}
\]
in the space of connections
$\bonf(Z)_{w}$: the quotient by the determinant-1 gauge group of the
space of all $\PU(2)$ connections in $\PP(E_{w})$, where $E_{w}\to Z$
is a $U(2)$ bundle with $\det(E)=w$. If the space of
these flat connections in $\bonf(Z)_{w}$ is non-degenerate in the
Morse-Bott sense when regarded as the set of critical points of the
Chern-Simons functional, then we have
\[
                    \dim I_{*}(Z)_{w} \le \dim H_{*}(\Rep(Z)_{w}).
\]
The generalized eigenspace $I_{*}(Z)_{w,+2}\subset I_{*}(Z)_{w}$ has
half the dimension of the total, so
\[
                    \dim \KHI(K) \le \frac{1}{2} \dim H_{*}(\Rep(Z)_{w}).
\]

As explained in \cite{KM-sutures}, the representation variety
$\Rep(Z)_{w}$ is closely related to the space
    \[
            \Rep(K,\bi) = \{ \, \rho: \pi_{1}(S^{3}
            \setminus K) \to \SU(2) \mid \rho(m) =
            \bi \,\},
    \]
    where  $m$ is a chosen meridian and
\[
            \bi = 
            \begin{pmatrix}
                i & 0 \\ 0 & -i
            \end{pmatrix}.
\]
More particularly, there is a two-to-one covering map
\begin{equation}\label{eq:covering}
            \Rep(Z)_{w} \to \Rep(K,\bi).
\end{equation}
The circle subgroup  $\SU(2)^{\bi}\subset \SU(2)$ which stabilizes $\bi$ acts on
$\Rep(K,\bi)$ by conjugation. There is a unique reducible element in
$\Rep(K,\bi)$ which is fixed by the circle action; the remaining
elements are irreducible and have stabilizer $\pm 1$. The most
non-degenerate situation that can arise, therefore, is that
$\Rep(K,\bi)$ consists of a point (the reducible) together with
finitely many circles, each of which is Morse-Bott. In such a case,
the covering \eqref{eq:covering} is trivial. As in
\cite{KM-knot-singular}, the corresponding non-degeneracy condition at
a flat connection $\rho$ can be interpreted as the condition that the
map
\[
                H^{1}(S^{3}\setminus K; \g_{\rho}) \to H^{1}(m ;
                \g_{\rho}) = \R
\]
is an isomorphism. Here $\g_{\rho}$ is the local system on the knot
complement with fiber $\su(2)$, associated to the representation
$\rho$. We therefore have:

\begin{corollary}
    Suppose that the representation variety $\Rep(K,\bi)$ associated
    to the complement of a classical knot $K\subset S^{3}$ consists of
    the reducible representation and $n(K)$ conjugacy classes of
    irreducibles, each of which is non-degenerate in the above sense.
    Then
    \[
            \dim \KHI(K) \le 1 + 2n(K).
    \]
\end{corollary}

\begin{proof}
    Under the given hypotheses, the representation variety
    $\Rep(K,\bi)$ is a union of a single point and $n(K)$ circles. Its
    total Betti number is therefore $1 + 2 n(K)$. The representation
    variety $\Rep(Z)_{w}$ is a trivial double cover \eqref{eq:covering},
    so the total Betti number of $\Rep(Z)_{w}$ is twice as large, $2 +
    4n(K)$.
\end{proof}

Combining this with Corollary~\ref{cor:alexander-vs-rank}, we obtain:

\begin{corollary}
    Under the hypotheses of the previous corollary, we have
    \[
                \sum_{j=-d}^{d} |a_{j}| \le 1 + 2n(K)
    \]
    where the $a_{j}$ are the coefficients of the Alexander
    polynomial.     \qed
\end{corollary}

Among all the irreducible elements of $\Rep(K,\bi)$, we can
distinguish the subset consisting of those $\rho$ whose image is
binary dihedral: contained, that is, in the normalizer of a circle
subgroup whose infinitesimal generator $J$ satisfies
$\mathrm{Ad}(\bi)(J)=-J$. 
If $n'(K)$ denotes the number of such
irreducible binary dihedral representations, then one has
\[
                   | \det(K) | = 1 + 2n'(K).
\]
(see \cite{Klassen}). On the other hand, the determinant
$\det(K)$ can also be computed as the value of the Alexander
polynomial at $-1$: the alternating sum of the coefficients. Thus we
have:

\begin{corollary}
        Suppose that the Alexander polynomial of $K$ fails to be
        alternating, in the sense that
        \[
                    \left| \sum_{j=-d}^{d} (-1)^{j} a_{j} \right|
                    <  \sum_{j=-d}^{d} | a_{j}|.
        \]
        Then either $\Rep(K,\bi)$ contains some representations that
        are not binary dihedral, or some of the binary-dihedral
        representations are degenerate as points of this
        representation variety. \qed
\end{corollary}

This last corollary is nicely illustrated by the torus knot
$T(4,3)$. This knot is the first non-alternating knot in Rolfsen's
tables \cite{Rolfsen}, where it appears as $8_{19}$. The Alexander
polynomial of $8_{19}$ is not alternating in the sense of the
corollary; and as the corollary suggests, the representation variety
$\Rep(8_{19}; \bi)$ contains representations that are not binary
dihedral. Indeed, there are representations whose image is the binary
octahedral group in $\SU(2)$.

\bibliographystyle{abbrv}
\bibliography{alexander}

\end{document}